\documentclass[reqno,a4paper]{amsart}
\usepackage{amsmath}
\usepackage{amssymb,amsfonts}
\usepackage{amsthm}
\usepackage{enumerate}
\usepackage[mathscr]{eucal}
\usepackage{xcolor}

\theoremstyle{plain}
\newtheorem{theorem}{Theorem}[section]

\theoremstyle{definition}

\newtheorem{remark}{\textnormal{\textbf{Remark}}}

\theoremstyle{remark}

\numberwithin{equation}{section}

\begin{document}

\title[Sequences on models of elliptic curves]
{Rational sequences on different models of elliptic curves}

\author{ Gamze Sava\c{s} \c{C}EL\.{I}K, Mohammad Sadek and G\"{o}khan Soydan}

\address{{\bf Gamze Sava\c{s} \c{C}elik}\\
Department of Mathematics \\
Bursa Uluda\u{g} University\\
 16059 Bursa, Turkey}
\email{gamzesavas91@gmail.com; gamzesavascelik@gmail.com }

\address{{\bf Mohammad Sadek}\\
	Faculty of Engineering and Natural Sciences\\ Sabanc{\i} University\\
  Tuzla, \.{I}stanbul, 34956 Turkey}
\email{mmsadek@sabanciuniv.edu}

\address{{\bf G\"{o}khan Soydan} \\
	Department of Mathematics \\
	Bursa Uluda\u{g} University\\
	16059 Bursa, Turkey}
\email{gsoydan@uludag.edu.tr }

\newcommand{\acr}{\newline\indent}


\begin{abstract}
Given a set $S$ of elements in a number field $k$, we discuss the existence of planar algebraic curves over $k$ which possess rational points whose $x$-coordinates are exactly the elements of $S$. If the size $|S|$ of $S$ is either $4,5$, or $6$, we exhibit infinite families of (twisted) Edwards curves and (general) Huff curves for which the elements of $S$ are realized as the $x$-coordinates of rational points on these curves. This generalizes earlier work on progressions of certain types on some algebraic curves.
\end{abstract}

\maketitle

\section{Introduction}\label{sec:1}
An algebraic (affine) plane curve $C$ of degree $d$ over some field $k$ is defined by an equation of the form
\[\{(x,y)\in k^2: f(x,y)=0\}\] where $f $ is a polynomial of degree $d$. The algebraic affine plane curve $C$ can also be extended to the projective plane by homogenising the polynomial $f$. If $P=(x,y)$, then we write $x=x(P)$ and $y=y(P)$.

Studying the set of $k$-rational points on $C$, $C(k)$, has been subject to extensive research in arithmetic geometry and number theory, especially when $k$ is a number field.  For example, if $f$ is a polynomial of degree $2$, then one knows that $C$ is of genus $0$, and so if $C$ possesses one rational point then it contains infinitely many such points. If $f$ is of degree $3$, then $C$ is a genus $1$ curve if it is smooth. In this case, if $C(k)$ contains one rational point, then it is an elliptic curve, and according to Mordell-Weil Theorem, $C(k)$ is a finitely generated abelian group. In particular, $C(k)$ can be written as $T\times \mathbb{Z}^r$ where $T$ is the subgroup of points of finite order, and $r\ge0$ is the {\em rank} of $C$ over $k$.

In enumerative geometry, one may pose the following question. Given a set of points $S$ in $k^2$, how many algebraic plane curve $C$ of degree $d$ satisfy that $S\subseteq C(k)$? It turns out that sometimes the answer is straightforward. For example, given $10$ points in $k^2$, in order for a cubic curve to pass through these points, a system of $10$ linear equations will be obtained by substituting the points of $S$ in \[a_1x^3+a_2x^2y+a_3x^2+a_4xy^2+a_5xy+a_6x+a_7y^3+a_8y^2+a_9y+a_{10}=0\] and solving for $a_1,\cdots,a_{10}$. Therefore, there exists a unique nontrivial solution to the system if the determinant of the corresponding matrix of coefficients is zero, hence a unique cubic curve through the points of $S$. Thus, one needs linear algebra to check the existence of algebraic curves of a certain degree through various specified points in $k^2$.

In this article, we address the following, relatively harder, question. Given $S\subset k$, are there algebraic curves $C$ of degree $d$ such that for every $x\in S$, $x=x(P)$ for some $P\in C(k)$? In other words, $S$ constitutes the $x$-coordinates of a subset of $C(k)$. The latter question can be reformulated to involve $y$-coordinates instead of $x$-coordinates. It is obvious that linear algebra cannot be utilized to attack the problem as substituting with the $x$-values of $S$ will not yield linear equations.

Given a set $S=\{x_1,x_2,\cdots,x_n\}\subset k$, if $(x_i,y_i)$, $i=1,\cdots,n$, are $k$-rational points on an algebraic curve $C$, then these rational points are said to be an {\em $S$-sequence} of length $n$. In what follows, we summarize the current state of knowledge for different types of $S$.

We first describe the state-of-art when the elements of $S\subset\mathbb{Q}$ are chosen to form an {\em arithmetic progression}, Lee and V\'{e}lez \cite{LV} found infinitely many curves described by $y^{2}=x^{3}+a$ containing $S$-sequences of length $4$. Bremner \cite{Br} showed that there are infinitely many elliptic curves with $S$-sequences of length $7$ and $8$. Campbell \cite{Cam} gave a different method to produce infinite families of elliptic curves with $S$-sequences of length $7$ and $8$. In addition, he described a method for obtaining infinite families of quartic elliptic curves with $S$-sequences of length $9$, and gave an example of a quartic elliptic curve with an $S$-sequence of length $12$. Ulas \cite{Ul} first described a construction method for an infinite family of quartic elliptic curves on which there exists an $S$-sequence of length $10$. Secondly he showed that there is an infinite family of quartics containing $S$-sequences of length $12$. Macleod \cite{Mac} showed that simplifying Ulas' approach may provide a few examples of quartics with $S$-sequences of length $14$. Ulas \cite{U2} found an infinite family of genus two curves described by $y^2=f(x)$ where $\deg(f(x))=5$ possessing $S$-sequences of length $11$. Alvarado \cite{Al} showed the existence of an infinite family of such curves with $S$-sequences of length $12$. Moody \cite{Mod1} found an infinite number of Edwards curves with an $S$-sequence of length $9$. He also asked whether any such curve will allow an extension to an $S$-sequence of length $11$. Bremner \cite{BrE} showed that such curves do not exist. Also, Moody \cite{Mod3} found an infinite number of Huff curves with $S$-sequences of length $9$, and Choudhry \cite{Ch} extended Moody's result to find several Huff curves with $S$-sequences of length $11$.

Now we consider the case when the elements of $S$ form a {\em geometric progression}, Bremner and Ulas \cite{BU} obtained an infinite family of elliptic curves with $S$-sequences of length $4$, and they also pointed out infinitely many elliptic curves with $S$-sequences of length $5$.  Ciss and Moody \cite{Mod2} found infinite families of twisted Edwards curves with $S$-sequences of length $5$ and Edwards curves with $S$-sequences of length $4$. When the elements of $S\subset\mathbb Q$ are {\em consecutive squares}, Kamel and Sadek \cite{KS} constructed infinitely many elliptic curves given by the equation $y^2=ax^3+bx+c$
with $S$-sequences of length $5$. When the elements of $S\subset\mathbb Q$ are {\em consecutive cubes}, \c{C}elik and Soydan \cite{SCS} found infinitely many elliptic curves of the form $y^2=ax^3+bx+c$ with $S$-sequences of length $5$.

In the present work, we consider the following families of elliptic curves due to the symmetry enjoyed by the equations defining them: (twisted) Edwards curves and (general) Huff curves. Given an arbitrary subset $S$ of a number field $k$, we tackle the general question of the existence of infinitely many such curves with an $S$-sequence when there is no restriction on the elements of $S$. We provide explicit examples when the length of the $S$-sequence is $4, 5$, or $6$. This is achieved by studying the existence of rational points on certain quadratic and elliptic surfaces.

\section{Edwards curves with $S$-sequences of length $6$}
\label{sec:1}
Throughout this work, $ k$ will be a number field unless otherwise stated.

An {\em Edwards curve} over $k$ is defined by \begin{equation} \label{Ed1}
E_{d}: x^2+y^2=1+dx^2y^2,
\end{equation}
where $d$ is a non-zero element in $k$. It is clear that the points $(x,y)=(-1,0),(0,\pm1),(1,0)\in E_{d}(k)$. We show that given any set \[S=\{s_{-1}=-1,s_0=0,s_1=1,s_2,s_3,s_4\}\subset k,\] $s_i\ne s_j$ if $i\ne j$, there are infinitely many Edwards curves $E_d$ that possess rational points whose $x$-coordinates are $s_i$, $-1\le i\le 4$, i.e., the set $S$ is realized as $x$-coordinates in $E_d(k)$. In other words, there are infinitely many Edwards curves that possess an $S$-sequence.

We start with assuming that $s_2$ is the $x$-coordinate of a point in $E_d(k)$, then one must have
$\displaystyle y^2=\frac{s_2^2-1}{s_2^2 d-1}$, or $s_2^2 d-1 =(s_2^2-1) p^2$ for some $p=1/y\in k$.

Similarly, if $s_3$ is the $x$-coordinate of a point in $E_d(k)$, then $\displaystyle y^2=\frac{s_3^2-1}{s_3^2 d-1}$, or $s_3^2 d-1=(s_3^2-1) q^2$. So
\[d=\frac{(s_2^2-1) p^2+1}{s_2^2}=\frac{(s_3^2-1) q^2+1}{s_3^2}.\]
Thus we have the following quadratic curve
\[s_3^2\left[(s_2^2-1) p^2+1\right]-s_2^2\left[(s_3^2-1) q^2+1\right]=0\] on which we have the rational point $(p,q)=(1,1)$. Parametrizing the rational points on the latter quadratic curve yields
\begin{eqnarray*}
	p &=& \frac{2 t s_2^2 - t^2 s_2^2 - s_3^2 + s_2^2 s_3^2 - 2 t s_2^2 s_3^2 + t^2 s_2^2 s_3^2}{-t^2 s_2^2 + s_3^2 - s_2^2 s_3^2 + t^2 s_2^2 s_3^2},\\
	q &=& -\frac{(-1 + s_2^2) s_3^2 - 2 t (-1 + s_2^2) s_3^2 +t^2 s_2^2 (-1 + s_3^2)}{-(-1 + s_2^2) s_3^2 + t^2 s_2^2 (-1 + s_3^2)}.
\end{eqnarray*}
Therefore, fixing $s_2$ and $s_3$ in $k$, one sees that $p$ and $q$ lie in $k(t)$.
Now we obtain the following result.

\begin{theorem}
	\label{thm1}
	Let $s_{-1}=-1$, $s_0=0$, $s_1=1$, $s_2$, $s_3$ and $s_4$, $s_i\ne s_j$ if $i\ne j$, be a sequence in $\mathbb Z$ such that \[h(s_2,s_3)=-3 + 4 s_3^2 + s_2^4 s_3^4 + s_2^2 (4 - 6 s_3^2)\ne0\] where either $g_1(s_2,s_3)/h(s_2,s_3)^2$ or $g_2(s_2,s_3)/h(s_2,s_3)^3$ are not integers, $g_1$ and $g_2$ are defined in (\ref{eq11}). There are infinitely many Edwards curves described by
	\[E_{d}: x^2+y^2=1+d x^2y^2,\quad d\in \mathbb Q\]
	on which $s_i$, $-1\le i\le 4$, are the $x$-coordinates of rational points in $E_{d}(\mathbb Q)$. In other words, there are infinitely many Edwards curves that possess an $S$-sequence where $S=\{s_i:-1\le i\le 4\}.$
\end{theorem}

\begin{proof}
	Substituting the value for $p$ in $\displaystyle d=\frac{(s_2^2-1) p^2+1}{s_2^2}$ yields that
	\begin{align*}
	(-t^2 s_2^2 + s_3^2 - s_2^2 s_3^2 + t^2 s_2^2 s_3^2)^2d&=(s_3^4 - 2 s_2^2 s_3^4 + s_2^4 s_3^4)+(4 s_3^2 - 8 s_2^2 s_3^2 + 4 s_2^4 s_3^2\\
	& - 4 s_3^4 + 8 s_2^2 s_3^4 - 4 s_2^4 s_3^4)t+( -4 s_2^2 + 4 s_2^4 - 4 s_3^2 \\
	&+ 14 s_2^2 s_3^2 -10 s_2^4 s_3^2 + 4 s_3^4 - 10 s_2^2 s_3^4 + 6 s_2^4 s_3^4)t^2\\
	&+(4 s_2^2 - 4 s_2^4 - 8 s_2^2 s_3^2 + 8 s_2^4 s_3^2 + 4 s_2^2 s_3^4\\
	&-4 s_2^4 s_3^4)t^3+(s_2^4 - 2 s_2^4 s_3^2 + s_2^4 s_3^4)t^4.
	\end{align*}
	Thus, for fixed values of $s_2$ and $s_3$, we have $d\in \mathbb Q(t)$.
	
	Now we show the existence of infinitely many values of $t$ such that $s_4$ is the $x$-coordinate of a rational point on $E_d$. In fact, we will show that $t$ can be chosen to be the $x$-coordinate of a rational point on an elliptic curve with positive Mordell-Weil rank, hence the existence of infinitely many such possible values for $t$. Forcing $(s_4,r)$ to be a point in $E_d(\mathbb Q)$ for some rational $r$ yields that  \begin{equation}\label{eq:1}r^2=\frac{s_4^2-1}{s_4^2d-1}=(A_0+A_1 t+A_2t^2+A_3t^3+A_4 t^4)/B(t)^2,\end{equation}
	where $A_i\in \mathbb Z$ and $B(t)=-t^2 s_2^2  + t^2 s_2^2 s_3^2+ s_3^2 - s_2^2 s_3^2.$ This implies that $A_0+A_1 t+A_2t^2+A_3t^3+A_4 t^4$ must be a rational square.
	This yields the elliptic curve $C$ defined by \[z^2=A_0+A_1 t+A_2t^2+A_3t^3+A_4 t^4,\] with the following rational point \[(t,z)=\left(0, s_3^2(s_2^2-1) \right).\] The latter elliptic curve is isomorphic to the elliptic curve described by the Weierstrass equation
	$E_{I,J}:y^2=x^3-27 Ix-27 J$ where
	\begin{eqnarray*}
		I&=& 12 A_0A_4 - 3 A_1A_3 + A_2^2 \\
		J&=& 72 A_0A_2A_4 + 9 A_1A_2A_3 - 27 A_1^2A_4 - 27A_0A_3^2- 2 A_2^3,
	\end{eqnarray*}
	see for example \cite[\S 2]{Stoll}. The latter elliptic curve has the following rational point
	\[P=\left(-12 (-1 + s_2^2) (-1 + s_3^2) (-3 + s_2^2 + s_3^2),-216 (-1 + s_2^2)^2 (-1 + s_3^2)^2\right).\]
	One notices that the coordinates of $3P$ are rational functions. Indeed,
	\begin{eqnarray}\label{eq11}3P=\left(\frac{g_1(s_2,s_3)}{h(s_2,s_3)^2},\frac{g_2(s_2,s_3)}{h(s_2,s_3)^3}\right),\qquad\textrm{ where } g_1,g_2\in\mathbb Q[s_2,s_3]\end{eqnarray} and
	$$h(s_2,s_3)=-3 + 4 s_3^2 + s_2^4 s_3^4 + s_2^2 (4 - 6 s_3^2).$$ Hence, as long as $h(s_2,s_3)\ne 0$, and $g_1/h^2\not\in\mathbb Z$ or $g_2/h^3\not\in\mathbb Z$, one sees that $3P$ is a point of infinite order by virtue of Lutz-Nagell Theorem. Thus, $P$ itself is a point of infinite order. It follows that $E_{I,J}$ is of positive Mordell-Weil rank.
	Since $C$ is isomorphic to $E_{I,J}$, it follows that $C$ is also of positive Mordell-Weil rank. Therefore, there are infinitely many rational points $(t,z)\in C(\mathbb Q)$, each giving rise to a value for $d$, by substituting in (\ref{eq:1}), hence an Edwards curve $E_d$ possessing the aforementioned rational points. That infinitely many of these curves are pairwise non-isomorphic over $\mathbb Q$ follows, for instance, from Proposition 6.1 in \cite{Edwards}.
\end{proof}

\section{ Twisted Edwards curves with $S$-sequences of length $4$}

A {\em Twisted Edwards curve} over $k$ is given by \begin{equation} \label{Ed2}
E_{a,d}: ax^2+y^2=1+dx^2y^2,
\end{equation}
where $a$ and $d$ are nonzero elements in $k.$ Note that the point $(x,y)=(0,\pm1)\in E_{a,d}(k)$. Given a set $\{u_0=0, u_1,u_2,u_3\}\subset k$, $u_i\ne u_j$ if $i\ne j$,
we prove that there are infinitely many twisted Edwards curves $E_{a,d}$ for which $S$ is realized as the $x$-coordinates of rational points on $E_{a,d}$.

We begin by assuming that $u_1$ is the $x$-coordinate of a point in $E_{a,d}(k)$, then one must get
$\displaystyle y^2=\frac{au_1^2-1}{u_1^2 d-1}$, or $u_1^2 d-1 =(au_1^2-1) i^2$ for some $i\in k$.

Now, if $u_2$ is the $x$-coordinate of a point in $E_{a,d}(k)$, then $\displaystyle y^2=\frac{au_2^2-1}{u_2^2 d-1}$ or $u_2^2 d-1 =(au_2^2-1) j^2$. So
\[d=\frac{(au_1^2-1) i^2+1}{u_1^2}=\frac{(au_2^2-1) j^2+1}{u_2^2}.\]
Hence we obtain the following quadratic surface
\[u_2^2\left[(au_1^2-1) i^2+1\right]-u_1^2\left[(au_2^2-1) j^2+1\right]=0,\] on which we have the rational point $(i,j)=(1,1)$. Solving the above quadratic surface gives the following
\begin{eqnarray*}
	i&=&\frac{-au_1^2u_2^2+u_2^2+2tau_1^2u_2^2-2tu_1^2-at^2u_1^2u_2^2+u_1^2t^2}{au_1^2u_2^2-u_2^2-at^2u_1^2u_2^2+u_1^2t^2},\\
	j&=& \frac{-2atu_1^2u_2^2+2tu_2^2+at^2u_1^2u_2^2-u_1^2t^2+au_1^2u_2^2-u_2^2}{au_1^2u_2^2-u_2^2-at^2u_1^2u_2^2+u_1^2t^2}.
\end{eqnarray*}

Now we get the following result.
\begin{theorem}
	Let $u_0=0$, $u_1$, $u_2$ and $u_3$ , $u_i\ne u_j$ if $i\ne j$, be a sequence in $\mathbb Z$ such that $h(u_1,u_2)\ne 0$, and either $g_1(s_2,s_3)/h(s_2,s_3)^2$ or $g_2(s_2,s_3)/h(s_2,s_3)^3$ are not integers, where $h,g_1,g_2$ are defined in (\ref{eq3}). There are infinitely many twisted Edwards curves described by
	\[E_{a,d}: ax^2+y^2=1+d x^2y^2,\quad d\in\mathbb{Q},\;a\in\mathbb{Q}^{\times} \textrm{ is arbitrary}\]
	on which $u_i$, $0\le i\le 3$, are the $x$-coordinates of rational points in $E(\mathbb Q)$. In other words, there are infinitely many twisted Edwards curves that possess an $S$-sequence where $S=\{u_i:0\le i\le 3\}.$
\end{theorem}
\begin{proof}
	Substituting the expression for $i$ in $\displaystyle d=\frac{(au_1^2-1) i^2+1}{u_1^2}$ gives that
	\begin{align*}
	(au_1^2u_2^2-u_2^2-at^2u_1^2u_2^2+u_1^2t^2)^2d&=(u_1^4a^3u_2^4-2u_1^4a^2u_2^2+u_1^4a)t^4+(-8au_1^2u_2^2\\
	&+4u_1^2+4u_1^2a^2u_2^4-4u_1^4a-4u_1^4a^3u_2^4\\
	&+8u_1^4a^2u_2^2)t^3+(-4u_1^2-10u_1^2a^2u_2^4+14au_1^2u_2^2\\
	&+6u_1^4a^3u_2^4-4u_2^2-10u_1^4a^2u_2^2+4u_1^4a+4au_2^4)t^2\\
	&+(4u_2^2+8u_1^2a^2u_2^4-8au_1^2u_2^2+4u_1^4a^2u_2^2-4au_2^4\\
	&-4u_1^4a^3u_2^4)t+u_1^4a^3u_2^4-2u_1^2a^2u_2^4+au_2^4.
	\end{align*}
	Then, assuming $(u_3,\ell)\in E(\mathbb Q)$ yields
	\begin{equation}\label{eq:2}
	\ell^2=\frac{au_3^2-1}{du_3^2-1}=(C_0+C_1 t+C_2t^2+C_3t^3+C_4 t^4)/D(t)^2,
	\end{equation}
	where $C_i\in \mathbb Q$ and $D(t)=au_1^2u_2^2-u_2^2-at^2u_1^2u_2^2+u_1^2t^2.$
	
	For the latter equation to be satisfied, one needs to find rational points on the elliptic curve $C'$ defined by \[z^2=C_0+C_1 t+C_2t^2+C_3t^3+C_4 t^4\] that possesses the rational point \[(t,z)=\left(0, u_2^2(au_1^2-1) \right).\] The latter elliptic curve is isomorphic to the elliptic curve described by the Weierstrass equation
	$E_{I,J}:y^2=x^3-27 Ix-27 J$ where
	\begin{eqnarray*}
		I&=& 12 C_0C_4 - 3 C_1C_3 + C_2^2, \\
		J&=& 72 C_0C_2C_4 + 9 C_1C_2C_3 - 27 C_1^2C_4 - 27C_0C_3^2- 2 C_2^3,
	\end{eqnarray*}
	see for example \cite[\S 2]{Stoll}. The latter elliptic curve has the following rational point
	\[Q=\left(-12 (-1 + au_2^2) (-1 + au_1^2) (-3 + au_2^2 + u_1^2),-216 (-1 + au_2^2)^2 (-1 + au_1^2)^2\right).\]
	One notices that the coordinates of $3Q$ are rational functions. In fact, $$3Q=\left(\frac{g_1(u_1,u_2)}{h(u_1,u_2)^2},\frac{g_2(u_1,u_2)}{h(u_1,u_2)^3}\right),\qquad\textrm{ where } g_1,g_2\in\mathbb Q[u_1,u_2]$$ and
	\begin{align}
	\label{eq3}
	h(u_1,u_2)&=-27 - 72 u_1^2 + 36 u_1^4 + 18 u_1^2 u_2^2 - 12 u_1^4 u_2^2 - 18 u_2^4 +
	12 u_1^2 u_2^4 + u_1^4 u_2^4\nonumber\\
	& - 2 u_1^2 u_2^6 + u_2^8+  a (36 u_1^2 - 12 u_1^4 - 24 u_1^2 (-3 + u_1^2) + 36 u_2^2 +
	72 u_1^2 u_2^2 \nonumber\\
	&- 24 u_1^4 u_2^2 - 12 u_1^2 u_2^4 + 4 u_1^4 u_2^4 -
	4 (-3 + u_1^2) u_2^6)+a^2 (-144 u_1^2 u_2^2 \nonumber\\
	&+ 36 u_1^4 u_2^2 + 18 u_2^4 - 36 u_1^2 u_2^4 +
	4 u_1^4 u_2^4 + 2 u_1^2 u_2^6 - 2 u_2^8)+ a^3 (36 u_1^2 u_2^4 \nonumber\\
	&+ 4 (-3 + u_1^2) u_2^6) +a^4 u_2^8.\nonumber\\
	\end{align}

	Therefore, as long as $h(u_1,u_2)\ne 0$ and $g_1/h^2\not\in\mathbb Z$ or $g_2/h^3\not\in\mathbb Z$, one sees that $E_{I,J}$ is of positive Mordell-Weil rank where the point $Q$ is of infinite order.
	Since $C'$ is isomorphic to $E_{I,J}$, it follows that $C'$ is also of positive Mordell-Weil rank. 	Hence, there are infinitely many rational points $(t,z)\in C'(\mathbb Q)$, each giving rise to a value for $d$, by substituting in (\ref{eq:2}), therefore a twisted Edwards curve $E_{a,d}$ possessing the aforementioned rational points. That infinitely many of these curves are pairwise non-isomorphic over $\mathbb Q$ again follows from Proposition 6.1 in \cite{Edwards}.
	
\end{proof}
\begin{remark}
	Since $(0,-1), (0,1)$ are rational points on any twisted Edwards curve, one can show that if $u_{-1}=-1$, $u_1=1$, $u_2$, $u_3$ and $u_4$, $u_i\ne u_j$ if $i\ne j$, is a sequence in $\mathbb Z$, there are infinitely many Edwards curves
	on which $u_i$, $ i\in\{-1,1,2,3,4\}$, are the $y$-coordinates of rational points in $E_{a,d}(\mathbb Q)$.
\end{remark}

\section{ Huff curves with $S$-sequences of length $5$}

A {\em Huff curve} over a number field $k$ is defined by \begin{equation}\label{Huff1}
H_{a,b}: ax(y^2-1)=by(x^2-1),
\end{equation}
with $a^2\ne b^2$. Note that the points $(x,y)=(-1,\pm1),(0,0),(1,\pm1)$ are in $H_{a,b}(k)$. We prove that given $s_{-1}=-1$, $s_{0}=0$, $s_{1}=1$, $s_2,$ $s_3$ $\in k$, $s_{i}\neq s_{j}$ if $i\neq j$, there are infinitely many Huff curves on which these numbers are realized as the $x$-coordinates of rational points.

Assuming $(s_2,p)$ and $(s_3,q)$ are two points on $H_{a,b}$ yields
\begin{equation}\label{eq:1.1}
as_2(p^2-1)=bp(s_2^2-1),
\end{equation}
and
\begin{equation}\label{eq:1.2}
as_3(q^2-1)=bq(s_3^2-1),
\end{equation}
respectively. Using \eqref{eq:1.1} and \eqref{eq:1.2}, one obtains
\[ \frac{s_2(p^2-1)}{s_3(q^2-1)}= \frac{p(s_2^2-1)}{q(s_3^2-1)},\] therefore, one needs to consider the curve
\begin{equation*}
C': Apq^2-Ap-Bqp^2+Bq=0,
\end{equation*} where $A=s_3s_2^2-s_2$ and $B=s_2s_3^2-s_2$. Dividing both sides of the above equality by $q^3$ gives
\begin{equation*}
A\frac{p}{q}-A\frac{p}{q}\frac{1}{q^2}-B(\frac{p}{q})^2+B\frac{1}{q^2}=0.
\end{equation*} Substituting  $x=\frac{p}{q}$ and $y=\frac{1}{q^2}$ in the above equation yields the following quadratic curve
\[Ax-Axy-Bx^2+By=0,\]
on which we have the rational point $(x,y)=(1,1)$. Parametrizing the rational points on the latter quadratic curve gives
\begin{equation}\label{xx}
x=\frac{Bt-B}{At+B},
\end{equation}
\begin{equation}\label{xxx}
y=\frac{At(1-t)+B(1-t)^2}{At+B}.
\end{equation}
Now we have the following result.
\begin{theorem}
	Let $s_{-1}=-1,s_{0}=0,s_{1}=1,s_2, s_3,$ $s_m\ne s_n$ if $m\ne n$, be a sequence in $\mathbb Z$ such that $$h=-4 + A^2 - 3 A B + B^2\ne0$$ where $A$ and $B$ are defined as above, and either $g_1/h^2$ or $g_2/h^3$ are not integers, where $g_1,g_2$ are defined in (\ref{eq22}). There are infinitely many Huff curves described by
	\[H_{a,b}: ax(y^2-1)=by(x^2-1),\qquad a,b\in \mathbb Q,\quad a^2\ne b^2\]
	on which $s_m$, $-1\le m\le 3$, are the $x$-coordinates of rational points in $H_{a,b}(\mathbb Q)$. In other words, there are infinitely many Huff curves that possess an $S$-sequence where $S=\{s_i:-1\le i\le 3\}.$
\end{theorem}
\begin{proof}
	Using the equalities \eqref{xx} and \eqref{xxx}, we obtain the following
	\begin{eqnarray*}
		p^2=\frac{x^2}{y}=\frac{B^2(-1+t)}{(B (-1 + t) - A t) (B + A t)}, \\
		q^2=\frac{1}{y}=\frac{(B + A t)}{(-1 + t) (B (-1 + t) - A t)}.
	\end{eqnarray*}
	In both cases we need
	$(B + A t)(-1 + t)(B (-1 + t) - A t)$ to be a square or in other words we need $t$ to be the $x$-coordinate of a rational point on the elliptic curve $C''$ defined by
	$$ z^2=(At+B)(t-1)(t(B-A)-B),$$
	with the following $k$-rational point $(t,z)=(0,B).$ The latter curve can be described by the following equation	
	$$ Y^2=X^3+((B-A)^2-AB)X^2-2AB(B-A)^2X+A^2B^2(B-A)^2,$$
	where $ A(B-A)t=X$ and $ A(B-A)z=Y$. This curve has the rational point \[R=(X,Y)=\left(0, AB(B-A) \right).\] Observing that  	
	\begin{eqnarray}\label{eq22}3R=\left(\frac{g_1(A,B)}{h(A,B)^2},\frac{g_2(A,B)}{h(A,B)^3} \right)\end{eqnarray} where $h(A,B)=-4 + A^2 - 3 A B + B^2$, one concludes as in the proof of Theorem \ref{thm1}.
\end{proof}

\section{General Huff curves with $S$-sequences of length $4$}

A {\em general Huff curve} over a number field $k$ is defined by \begin{equation}\label{Huff2}
G_{a,b}: x(ay^2-1)=y(bx^2-1),
\end{equation}
where $a,b\in k$ and $ab(a-b)\ne 0$. It is clear that the point $(x,y)=(0,0)\in G_{a,b}(k)$. We show that given $u_0=0,u_1,u_2,u_3$ in $k$, $u_i\ne u_j$ if $i\ne j$, there are infinitely many general Huff curves over which these points are realized as the $x$-coordinates of rational points.

We start by assuming that if $u_1$ is the $x$-coordinates of a point in $G_{a,b}(k)$, then one must have
$\displaystyle \frac{ay^2-1}{y}=\frac{bu_1^2-1}{u_1}$ or $\displaystyle \frac{a-i^2}{i}=\frac{bu_1^2-1}{u_1}$ for some $i\in k$.

Similarly, if $u_2$ is the $x$-coordinate of a point in $G_{a,b}$, then
$\displaystyle \frac{ay^2-1}{y}=\frac{bu_2^2-1}{u_2}$ or $\displaystyle \frac{a-j^2}{j}=\frac{bu_2^2-1}{u_2}$ for some $j\in k$. Thus, one obtains
$$ a=\frac{(bu_1^2-1)i+u_1i^2}{u_1}=\frac{(bu_2^2-1)j+u_2j^2}{u_2},$$ which gives the following quadratic curve
\begin{equation*}
S: Ai^2+Bj^2+Ciz+Djz=0,
\end{equation*}
where $A=-u_1u_2$, $B=u_1u_2$, $C=-u_1^2u_2b+u_2$, $D=bu_1u_2^2-u_1$. Then consider the line
\begin{equation*}
mP+nQ=(np:nq:m+nr)
\end{equation*}
connecting the rational points $P=(i:j:z)=(0: 0: 1)$ and $Q=(p: q: r)$ lying on $S\subset\mathbb{P}^2$. The intersection of $S$ and $mP+nQ$ yields the quadratic equation
$$n^2(Ap^2+Bq^2+Cpr+Dqr)+mn(Cp+Dq)=0.$$Using $P$ and $Q$ lying on $S$, one solves this quadratic equation and obtains formulae for the solution $(i:j:z)$ with the following parametrization:
\begin{eqnarray*}\label{eq:yy}
	i&=&np=Cp^2+Dpq,\\
	j&=&nq=Cpq+Dq^2,\\
	z&=&m+nr=-Ap^2-Bq^2.
\end{eqnarray*}


Now we obtain the following result.
\begin{theorem}
	Let $u_0=0$, $u_1$, $u_2$ and $u_3$ , $u_i\ne u_j$ if $i\ne j$, be a sequence in $k$. There are infinitely many general Huff curves described by
	$$G_{a,b}: x(ay^2-1)=y(bx^2-1),\qquad a,b\in k,\quad ab(a-b)\ne 0.
	$$
	on which $u_i$, $0\le i\le 3$, are the $x$-coordinates of rational points in $G_{a,b}(k)$. In other words, there are infinitely many general Huff curves that possess an $S$-sequence where $S=\{u_i:0\le i\le 3\}.$
\end{theorem}

\begin{proof}
	Substituting the value for $i$ in $\displaystyle a=\frac{(bu_1^2-1)i+u_1i^2}{u_1}$ yields that
	\begin{equation*}
	\begin{aligned}
	a=&u_{{2}}^2 \left( b{u_{{1}}}^{2}-1 \right) ^{2}{p}^{4}-2\,u_{1}u_2
	\left( b{u_{{2}}}^{2}-1 \right)  \left( b{u_{{1}}}^{2}-1
	\right) {p}^{3}q+{{u_{{1}}}^{2} \left( b{u_{{2}}}^{2}-1
		\right) ^{2}{p}^{2}{q}^{2}}\\
	&-\frac{u_2\left( b{u_{{1}}}^{2}-1
		\right) ^{2}}{u_1}{p}^{2}
	+ \left( b{u_{{2}}}^{2}-1 \right)
	\left( b{u_{{1}}}^{2}-1 \right) pq.
	\end{aligned}
	\end{equation*}
	Now we assume that $(u_3,\ell)\in G_{a,b}(k)$. This yields that \begin{equation*}
	\begin{aligned}
	&pu_{{3}} \left( b{p}^{2}{u_{{1}}}^{3}u_{{2}}-bpq{u_{{1}}}^{2}{u_{{2}}}
	^{2}-{p}^{2}u_{{1}}u_{{2}}+pq{u_{{1}}}^{2}-b{u_{{1}}}^{2}+1 \right)\\
	&\left( bp{u_{{1}}}^{2}u_{{2}}-bqu_{{1}}{u_{{2}}}^{2}-pu_{{2}}+qu_{{1}
	} \right) {\ell}^{2}-u_{{1}} \left( b{u_{{3}}}^{2}-1 \right) \ell-u_{1}u_{{3}}=0.
	\end{aligned}
	\end{equation*}
	This can be rewritten as
	\begin{equation*}
	\begin{aligned}
	&Z^2(b^2p^4u_1^5u_2^2u_3-2bp^4u_1^3u_2^2u_3-b^2p^2u_1^4u_2u_3+p^4u_1u_2^2u_3+2bp^2u_1^2u_2u_3\\
	&-p^2u_2u_3)+q Z (-2b^2p^3u_1^4u_2^3u_3+2bp^3u_1^4u_2u_3+2bp^3u_1^2u_2^3u_3+b^2pu_1^3u_2^2u_3\\
	&-2p^3u_1^2u_2u_3-bpu_1^3u_3-bpu_1u_2^2u_3+pu_1u_3)+
	q^2 p^2 u_1^3u_3(bu_2^2-1)^2\\&
	-T Z u_1 \left( b{u_{{3}}}^{2}-1 \right)-T^2u_1u_{{3}}=0,
	\end{aligned}
	\end{equation*}where $T=1/\ell.$ One sees that the rational point $P=(q:T:Z)=(1 :0:u_1 (-1 + b u_2^2)/pu_2(-1+bu_1^2))$ lies on the quadratic curve above, hence we may parametrize the rational points on the quadratic curve above. This is obtained by considering the intersection of the line $dP+eQ$ where $Q=(q_1:q_2:q_3)$ is a point on the quadratic curve. In fact, this yields that
	\begin{align*}
	d &= pu_{2}(b{u_{1}}^{2} - 1)({q_{3}}^{2}b^{2}p^{4}
	{u_{1}}^{5}{u_{2}}^{2}{u_{3}} - 2{q_{3}}^{2}bp^{4}{u
		_{1}}^{3}{u_{2}}^{2}{u_{3}}\\
	&-{q_{3}}^{2}b^{2}p^{2}{u_{
			1}}^{4}{u_{2}}{u_{3}}
	+ {q_{3}}^{2}p^{4}{u_{1}}{u_{2}}^{2}{u_{3}}+ 2
	{q_{3}}^{2}bp^{2}{u_{1}}^{2}{u_{2}}{u_{3}}\\&-{q_{3}}^{2
	}p^{2}{u_{2}}{u_{3}} - {u_{1}}{q_{2}}{q_{3}}b{u_{3}
	}^{2} + {u_{1}}{q_{2}}{q_{3}}
	+ p^{2}{u_{1}}^{3}{u_{3}}{q_{1}}^{2}b^{2}{u_{2}
	}^{4}\\&- 2p^{2}{u_{1}}^{3}{u_{3}}{q_{1}}^{2}b{u_{2}}^{
		2}+ p^{2}{u_{1}}^{3}{u_{3}}{q_{1}}^{2}- 2{q_{1}}{q_{3
	}}b^{2}p^{3}{u_{1}}^{4}{u_{2}}^{3}{u_{3}}\\&+ 2{q_{1}}{q_{3}}bp^{3}{u_{1}}^{4}{u_{2}}{u
		_{3}} + 2{q_{1}}{q_{3}}bp^{3}{u_{1}}^{2}{u_{2}}^{3}
	{u_{3}} + {q_{1}}{q_{3}}b^{2}p{u_{1}}^{3}{u_{2}}^{2}{
		u_{3}}
	\\&- 2{q_{1}}{q_{3}}p^{3}{u_{1}}^{2}{u_{2}}{u_{3
	}}
	- {q_{1}}{q_{3}}bp{u_{1}}^{3}{u_{3}} - {q_{1}}{q_{
			3}}bp{u_{1}}{u_{2}}^{2}{u_{3}} + {q_{1}}{q_{3}}p{
		u_{1}}{u_{3}}\\&-{u_{1}}{u_{3}}{q_{2}}^{2}) ,
	\end{align*}
	\begin{align*}
	e &= {u_{1}}(b{u_{2}}^{2} - 1)( - p{u_{1}}^{3}{u_{3}}{q
		_{1}}b^{2}{u_{2}}^{2} + p^{2}{u_{3}}{q_{3}}{u_{2}}b^{
		2}{u_{1}}^{4} + p{u_{1}}{u_{3}}{q_{1}}b{u_{2}}^{2}\\&
	- 2p^{2}{u_{3}}{q_{3}}{u_{2}}b{u_{1}}^{2} + p
	{u_{1}}^{3}{u_{3}}{q_{1}}b+ {u_{1}}{q_{2}}b{u_{3}}
	^{2} + p^{2}{u_{3}}{q_{3}}{u_{2}} - {u_{1}}{q_{2}}\\&- p{u_{1}}{u_{3}}{q_{1}}) .
	\end{align*}
\end{proof}

\subsection*{Acknowledgements}
We would like to thank the referees for carefully reading our manuscript and for giving such constructive comments which substantially helped improving the presentation of the paper.

\end{document}